\documentclass[a4paper,11pt,english]{amsart} \textheight21cm
\usepackage{babel,varioref}
\usepackage[ansinew]{inputenc}
\usepackage{amsmath}
\usepackage{amsfonts}
\usepackage{amssymb}
\usepackage{amscd}
\usepackage{enumerate}
\usepackage{latexsym}
\usepackage{graphics}
\usepackage{babel,varioref}
\usepackage{amscd}
\usepackage[all]{xypic}
\usepackage[pdftex]{graphicx}
\DeclareGraphicsExtensions{.pdf,.png,.jpg}

\input amssym.def
\input amssym.tex

\def\rul{\rule[-1.8mm]{0pt}{3.6ex}}

\def\reels{\mathbb{R}}
\def\complexes{\mathbb{C}}
\def\quaternions{\mathbb{H}}
\def\corps{\mathbb{K}}

\def\ad{\mbox{\textnormal{ad}}}
\def\Ad{\mbox{\textnormal{Ad}}}
\def\Id{\mbox{\textnormal{Id}}}
\def\lie#1{\g l\g i \g e({#1})}

\def\g {\mathfrak}

\def\alg{\g g}
\def\tialg{\tilde{\g g}}
\def\tip{\tilde{\g p}}
\def\tih{\tilde{\g h}}

\def\fracl#1#2{#1/#2}

\def\obullet{{\mbox{\makebox[1.5pt][l]{$\odot$}$\bullet$}}}

\newtheorem{theo}{Theorem}
\newtheorem{prop}[theo]{Proposition}
\newtheorem{corr}[theo]{Corollary}
\newtheorem{lem}[theo]{Lemma}

\newtheorem*{rem}{Remark}

\title[Infinitesimal semisimple symplectic extrinsic symmetric spaces]{Infinitesimal semisimple symplectic extrinsic symmetric spaces}
\author[Thomas Krantz]{Thomas Krantz}
\thanks{The author was supported by an AFR-grant of the ministry of
research of the G.-D. of Luxembourg and co-financed by the 'Marie Curie'-framework of the European Union.}

\makeindex
\begin{document}
\begin{abstract}
We study infinitesimal semi-simple extrinsic symmetric spaces and give a classification in the symplectic case.
\medskip


\noindent{\em MSC-2000:} 53C35, 53D05.
\end{abstract}
\maketitle

\section{Introduction}
Let $(M,s)$ be a symmetric space and let $a$ be a non degenerate reflexive form on $M$ invariant by the $s_x$ for $x\in M$. Let $\tilde M \hookrightarrow M$ be a submanifold. For every point $x\in \tilde M$ suppose that the tangent space $T_x \tilde M$ in $x$ is non degenerate w.r.t. $a$. Let $\mathcal{N}_x \tilde M=(T_x \tilde M)^{\perp_a}$ be the normal space in $x$. Suppose that for every point $x\in \tilde M$ there is an involutive automorphism $S_x$ of the symmetric space $(M,s)$ which verifies $S_x(x)=x$, $T_x(S_x)=-\Id_{T_x\tilde M} \oplus \Id_{\mathcal{N}_x \tilde M}$ and $S_x(\tilde M)=\tilde M$. We call $\tilde M$ an {\em extrinsic symmetric space} embedded in the symmetric
space $M$.

A big amount of work has been done in the case where the outer symmetric space is euclidian (\cite{KoNa}, \cite{Fe}, \ldots), pseudo-euclidian (\cite{Na0}, \cite{Ki}, \cite{KE}, \cite{Ka1}, \cite{Ka2}, \ldots) or symplectic affine(\cite{CGRS}, \cite{KS}, \ldots).

The general case where the outer space is a symmetric space has been less considered.
Riemannian extrinsic symmetric spaces have been classified by Naitoh and others, see the excellent overview in the book of Berndt-Console-Olmos. In his thesis N. Richard has considered the case where $a$ is a symplectic form (\cite{R}).

Let $G$ be the group of automorphisms of $(M,s)$ and $H\subset G$ the isotropy group fixing the point $o\in \tilde M$ so that $M=G/H$. Let $\tilde G$ be the group generated by the $S_x$ for $x\in\tilde M$ and $\tilde H$ be the subgroup of $\tilde G$ fixing the point $o$. In this article we will be interested more particularly in semi-simple extrinsic symmetric spaces {\em i.e.} where $G$ and $\tilde G$ are semi-simple.

We can consider the infinitesimal object associated to $\tilde M \hookrightarrow M$.
It is the data of the quintuple $(\g{g}, \tilde {\g{g}}, \sigma, \theta, a)$
where $\g{g}$ (resp. $\tilde {\g{g}}$) is the Lie algebra of $G$ (resp. $\tilde G$) {\em i.e.} $\g{g}=T_eG$ (resp.  $\tilde{\g{g}}=T_e\tilde G$),
$\sigma$ is $\Ad(s_o)$ and $\theta$ is $\Ad(S_o)$.

We will classify the semi-simple symplectic extrinsic symmetric quintuples.

{\sc Acknowledgements:} The author thanks Prof. Lorenz Schwachhöfer for support and helpful discussions, and {\em Technische Universität Dortmund} for hospitality from September 2010 to August 2011.
The author also acknowledges support from Prof. Martin Olbrich. In particular the characterization of symplectic extrinsic symmetric spaces as orbits of the isotropy representation of symplectic symmetric spaces has been obtained in collaboration.

\section{Semi-simple extrinsic symmetric quadruples}
\subsection{Basic definitions}
Let $\alg$ be a Lie algebra over the field $\corps=\reels$ or $\complexes$.
Let $\sigma$ be an involution of $\alg$, and $\g p$ (resp. $\g h$) the $-1$ (resp. $+1$) eigenspace of $\sigma$.
$(\g g, \sigma)$ is then called a {\em symmetric pair}. The Lie subalgebra $\g h$ is called the {\em isotropy algebra}.
It is known for a symmetric space $(M,s)$ that we can associate to it the symmetric pair $(\g g, \sigma)$, where $\g g$ is the Lie algebra associated to the Lie group $G$ generated by the $s_x$ for $x\in M$ and $\sigma$ is $\Ad(s_o)$. In this setting $\g h$ identifies to the Lie algebra of the isotropy group $H=\{g\in G \,|\, g\cdot o=o\}$ and $\g p$ to the tangent space $T_o M$.

According to $\corps=\complexes$ (resp. $\corps=\reels$), we call $(\g g, \sigma)$ a {\em complex} (resp. {\em real}) symmetric pair.
A real symmetric pair $(\g g, \sigma)$ is called {\em pseudo-complex} if it is obtained from a complex symmetric pair by restricting the field to $\reels$.

A symmetric pair is called {\em decomposable} if it is isomorphic to the direct product of two nontrivial symmetric pairs, {\em undecomposable} otherwise.

A symmetric pair $(\g g, \sigma)$ is called {\em (semi-)simple} (resp. {\em compact}), if $\g g$ is (semi-)simple (resp. compact).
A real simple symmetric pair is called {\em absolutely simple} if it is not pseudo-complex.

\subsection{The extrinsic setting}

Let $\theta$ and $\sigma$ be two commuting involutions of a semi-simple Lie algebra $\alg$.
Let $\g p$ (resp. $\g h$) be the $-1$ (resp. $+1$) eigenspace of $\sigma$. Let $\alg_-$ (resp. $\alg_+$) be the $-1$ (resp. $+1$) eigenspace of $\theta$. Define  $\g p_-:= \g p \cap \alg_-$, $\g p_+:= \g p \cap \alg_+$, $\g h_-:= \g h \cap \alg_-$, $\g h_+:= \g h \cap \alg_+$.
For $\lambda$ a $\corps$-linear mapping from $\g p_-$ to $\g h_-$, define $\tip := \langle X + \lambda(X) \,|\, X \in \g p_-\rangle$. Suppose $\tih := [\tip, \tip] \subset\g h$ and $[[\tip, \tip], \tip] \subset \tip$.
Let $\tialg:= \tip \oplus \tih$.
We call then $(\alg, \tialg, \sigma, \theta)$ an {\em extrinsic symmetric quadruple}.
Note that for an extrinsic symmetric quadruple the corresponding function $\lambda$ is uniquely defined.

For an extrinsic symmetric space as described in the introduction we obtain $\lambda$ ass follows:
The tangent space $T_o\tilde M$ identifies on one hand to the subspace $\g p_-$ of $T_o M$, on the other hand to $\tilde{\g p}$. This gives a mapping $X\mapsto X+\lambda(X)$ from $\g p_-$ to $\tilde{\g p}$.
Note that the equality $\lambda=0$ is equivalent to the property that the embedded extrinsic symmetric space is totally geodesic.

It is known that $\lambda$ satisfies the following properties:
$$[\lambda(x),y]=[\lambda(y),x]$$
and
$$\lambda([[x,y]+[\lambda(x),\lambda(y)],z])=[[x,y]+[\lambda(x),\lambda(y)],\lambda(z)],$$
for all $x,y,z\in\g p_-$.

As before an extrinsic symmetric quadruple is called {\em decomposable} if it is isomorphic to the direct product of two nontrivial extrinsic symmetric quadruples, {\em undecomposable} otherwise.

If moreover $\alg$ and $\tialg$ are semi-simple (resp. compact), $(\alg, \tialg, \sigma, \theta)$ is called a {\em semi-simple} (resp. {\em compact}) extrinsic symmetric quadruple.

\subsection{Invariant common Cartan subalgebra}\label{icsa}
Recall that a Cartan subalgebra (csa) $\g h_0$ of $\alg$ is a maximal nilpotent self-normalizing subalgebra of $\g g$ or, which is equivalent when $\alg$ is semi-simple, a maximal abelian subalgebra such that for every $x\in\g h_0$, $ad(x)$ acts semi-simply.

\begin{theo}\label{csa}
Let $(\alg, \tialg, \sigma, \theta)$ be a semisimple extrinsic symmetric quadruple over the field $\corps=\complexes$ such that
the rank\footnote{{\em i.e.} the dimension of the Cartan subalgebras} of $\tialg$ equals the rank of the isotropy subalgebra $\tih$ of $\tialg$.
Then  $\alg$ (resp. $\tialg$) admits a Cartan subalgebra $\g h_0$ (resp. $\tilde{\g h}_0$) such that $\tilde{\g h}_0\subset \g h_0$, $\g h_0$ is $\sigma$- and $\theta$-invariant and $\tilde{\g h}_0$ is $\theta$-invariant.
\end{theo}

Let us call $\g h_0$ in this case an {\em extrinsic Cartan subalgebra} associated to the data $(\alg, \tialg, \sigma, \theta)$.

\begin{proof}
We mainly follow the steps of the exercises 11 to 13 in Ch. VII, §3 of \cite{Bo}.

First we recall a result that can be found in (\cite{KN}, Vol. II, p.327).
\begin{lem}
If $\mathbf{G}$ is a compact group of automorphisms of a Lie algebra $\alg$,
$\alg$ admits a Levi-Malcev decomposition  $\alg=\g s \oplus \g{r}$ such that $\mathbf{G}$ leaves the Levi-subalgebra $\g s$ invariant.
\end{lem}
Note that the radical $\g{r}$ is also $\mathbf{G}$-invariant.

\begin{lem}\label{ginv}
Let $\alg$ be a Lie algebra and $\mathbf{G}$ a hypersolvable group of automorphisms of $\alg$, then there is a $\mathbf{G}$-invariant Cartan subalgebra $\g h$ of $\alg$. If $\alg=\g s \oplus \g{r}$ is a $\mathbf{G}$-invariant Levi-Malcev decomposition, $\g h$ can be found such that it is the direct sum of a $\mathbf{G}$-invariant csa of the Levi-subalgebra $\g s$ and a $\mathbf{G}$-invariant csa of the radical $\g{r}$.
\end{lem}
The proof of lemma~\ref{ginv} follows from several lemmas.

\begin{lem}
Let $\mathbf{G}$ be a finite group of automorphisms of a solvable Lie algebra $\alg$.
There exists a $\mathbf{G}$-invariant Cartan subalgebra $\g h$ of $\alg$.
\end{lem}
Proof by induction see \cite{Bo} ex. 13.

\begin{lem}
Let $\alg$ be a neither semisimple nor solvable Lie algebra.
Let $\mathbf{G}$ be a hypersolvable group of automorphisms of $\alg$, and let $\alg=\g s \oplus \g{r}$ be a $\mathbf{G}$-invariant Levi-Malcev decomposition.
If the semi-simple Lie algebra $\g s$ admits a $\mathbf{G}$-invariant Cartan subalgebra, then $\alg$ admits a $\mathbf{G}$-invariant Cartan subalgebra $\g h$ which is the sum of a $\mathbf{G}$-invariant csa of $\g s$ and a $\mathbf{G}$-invariant csa of $\g{r}$.

\end{lem}
By hypothesis $\g s$ admits a $\mathbf{G}$-invariant Cartan subalgebra $\g h_{\g s}$, let $\g c$ be the commutator in $\alg$ of $\g h_{\g s}$, which is also $\mathbf{G}$-invariant.
Consider a $\mathbf{G}$-invariant csa $\g h_{\g{r}}$ of $\g c \cap \g{r}$, which exists by the preceding lemma. Take $\g h:=\g h_{\g s}\oplus \g h_{\g{r}}$, which has the required properties.

\begin{lem}
Let $\g s$ be a semisimple Lie algebra and $\dim \g s=n$.
If $\mathbf{G}$ is a hypersolvable group of automorphisms of $\alg$, and if any Lie algebra $\alg$ with $\dim \alg<n$ admits a $\mathbf{G}$-invariant Cartan subalgebra, then $\g s$ admits a $\mathbf{G}$-invariant Cartan subalgebra.
\end{lem}
Like in ex. 12 consider a cyclic subgroup $\mathbf{C}$ of $\mathbf{G}$ which fixes a proper subalgebra $\g h$ of $\alg$. If such a $\mathbf{C}$ does not exist, take any csa of $\alg$ and we are finished.
Otherwise, by induction $\g h$ admits a $\mathbf{G}$-invariant csa $\g h_0$. The commutator $\g c$ of $\g h_0$ in $\alg$ is $\mathbf{G}$-invariant and strictly included in $\alg$. It admits by induction a $\mathbf{G}$-invariant csa $\g h_1$. The latter is also a csa of $\alg$.

Note that $\g h_1$ contains necessarily the center of $\g c$ which is $\g h_0$.
This finishes the proof of lemma 2.

To prove theorem~\ref{csa} we take $\mathbf{G}=\langle\sigma,\theta\rangle$ and make the following observations:
Starting from a semisimple extrinsic quadruple $(\alg, \tialg, \sigma, \theta)$ for which the rank of the isotropy algebra $\tih$ equals the rank of $\tialg$, we can chose a csa $\tih_0$ of $\tialg$ contained in $\tih$. The csa $\tih_0$ is $\mathbf{G}$-invariant as $\tih$ is contained in $\g h_+$. Consider the commutator $\g c$ of $\tih_0$ which is a proper $\mathbf{G}$-invariant subalgebra of $\alg$. By lemma~\ref{ginv} the Lie algebra $\g c$ admits a $\mathbf{G}$-invariant csa $\g h_0$ which contains $\tih_0$. $\g h_0$ is a $\mathbf{G}$-invariant csa of $\alg$.
\end{proof}

\section{The symplectic case}
\subsection{Definitions}
If $(\g g, \sigma)$ is a symmetric pair and $\omega: \alg \times \alg \to \corps$ a non-degenerate skew-symmetric bilinear form such that $\omega(\g h,\cdot)=0$ and
$\omega$ is $ad_{\g h}$-invariant, $(\g g, \theta, \omega)$ is then called a {\em symplectic} symmetric triple.

As before a symplectic symmetric triple is called {\em decomposable} if it is isomorphic to the direct product of two nontrivial symplectic symmetric triples, {\em undecomposable} otherwise.

A semisimple symmetric pair $(\alg,\sigma)$ is called {\em (pseudo\footnote{pseudo refers to the case when the Killing form $B$ is not positive definite}-, para-) hermitian} if there is an $\ad_{\g h}$-invariant endomorphism $J$ of $\g p$ such that $B(Ju,v)+B(u,Jv)=0$ for $u,v\in\g p$ and such that $J^2=\varepsilon\Id$, where $\varepsilon=-1$ in the (pseudo-) hermitian case and $\varepsilon=1$ in the para-hermitian case and $B$ is the Killing form of $\alg$.

\subsection{Facts on symplectic semi-simple symmetric triples}
Let us recall some facts about semisimple symplectic symmetric triples from \cite{Bi}.

A semisimple symplectic symmetric triple is undecomposable if and only if it is simple.

A simple symmetric triple $(\g g, \theta, \omega)$ is symplectic if and only if the dimension (w.r. to $\corps$) of the center of $\g h$ is $1$. In this case there is always an element $Z$ in the center of $\g h$ such that $\omega(X,Y)=-B(Z, [X,Y])$ for $X,Y\in \g p$.
For a simple symplectic symmetric triple $(\g g, \theta, \omega)$ there is a csa $\g h_0$ of $\g g$ contained in $\g h$ and containing $Z$. In particular $\alg$ and $\g h$ have the same rank.

We suppose now that $\corps=\complexes$.
With respect to the csa $\g h_0$ there is a system of roots $\Phi\subset\g h_0^*$, a system of simple roots $\Delta\subset\Phi$ and $\alpha_0\in\Delta=\{\alpha_i \, | \, i=0\ldots n-1\}$ with the following properties: If we denote by $(h_{\alpha_i})_i\in \g h_0^n$ the dual base of $(\alpha_i)_i$, there is a scalar $r\in\corps$ such that $Z=r h_{\alpha_0}$ and $\g h=\corps h_{\alpha_0} \oplus \g s$, where $\g s$ is a semi-simple Lie algebra with csa $\g h_{\g s}:=h_{\alpha_0}^\perp\cap \g h_0$ and simple roots $\{\alpha_1\,\vrule\,_{\g h_{\g s}}, \ldots, \alpha_{n-1}\,\vrule\,_{\g h_{\g s}}\}$. Furthermore every root $\alpha$ of $\g g$ can be written $\alpha=n_0\alpha_0+\sum_{i>0} n_i \alpha_i$ with $n_0=0$ or $\pm 1$. We call $(\g h_0, \Delta, \alpha_0, r)$ a {\em defining quadruple} of the simple symplectic triple $(\g g, \theta, \omega)$.
In the case $r=1$ we call $\omega$ the {\em canonical symplectic form} associated to the simple symplectic symmetric triple $(\g g, \theta, \omega)$. The notion generalizes in an obvious way to semi-simple symplectic symmetric triples.

A real absolutely simple symplectic symmetric triple is either (pseudo-) hermitian or para-hermitian (with $J\in\corps\ad_Z$), a real pseudo-complex symplectic symmetric triple is both.
If the triple is (pseudo-)hermitian (resp. para-hermitian) there is a maximally compact (resp. maximally non compact) (see \cite{Bi}, \cite{Knapp}) csa containing the center of the isotropy algebra and contained in the isotropy algebra.

\subsection{The extrinsic setting}
\subsubsection{Definitions}
Symplectic extrinsic symmetric spaces have been studied in \cite{R}.
If $(\alg, \tialg, \sigma, \theta)$ is an extrinsic symmetric quadruple, and there is a non-degenerate skew-symmetric form $\omega$ such that $(\alg, \sigma, \omega)$ and $(\tialg, \theta\,\vrule\,_{\tialg}, \tilde{\omega})$, with $\tilde{\omega}(X+\lambda(X),Y+\lambda(Y))=\omega(X,Y)$ for $X,Y\in\g p_-$, are symplectic symmetric triples, we call $(\alg, \tialg, \sigma, \theta, \omega)$ then a {\em symplectic} extrinsic symmetric quintuple.

\subsubsection{Results in the complex case}\label{resultscomplex}
In \cite{NT} H. Nakagawa and R. Takagi classify the compact minimal extrinsic Kähler symmetric spaces embedded in the complex projective space. It follows from the work of H. Naitoh (\cite{Na1}-\cite{Na6}) that these are the only compact Kähler extrinsic symmetric spaces. By Theorem \ref{complextorealform} we will obtain the complete list of complex semisimple symplectic extrinsic symmetric quintuples. In the complex case consequently the outer symmetric pair is necessarily of the form $\g s \g l(N+1,\complexes)/\g g\g l (N,\complexes)$.

Let $(\alg, \tialg, \sigma, \theta, \omega)$ be a complex semi-simple symplectic extrinsic symmetric quintuple. Throughout this section we will suppose that $(\alg,\sigma)$ corresponds to the symmetric pair $\g s \g l(N+1,\complexes)/\g g\g l(N,\complexes)$.

$\tilde{\mathcal{T}}:=(\tialg, \theta, \omega\,\vrule\,_{\tialg\times\tialg})$ is a semi-simple symplectic symmetric triple. Denote by $\tih$ the isotropy algebra.
We can choose an invariant csa $\tih_0\subset\tih$ containing the center of $\tih$.
Let for the given $\tih_0$ be a defining quadruple $(\tih_0, \tilde\Delta, \tilde\alpha_0, \tilde r)$ of $\tilde{\mathcal{T}}$.
By the preceding section is easy to see that we can extend $\tih_0$ to a $\theta$-invariant csa $\g h_0$ of $\g h$. As $(\alg, \sigma, \omega)$ is a simple symplectic triple we see that $\g h_0$ is also a csa of $\alg$. Denote $Z$ a generator of the center of $\g h$. If not $Z\in\g h_0$, we could extend $\g h_0$ to a bigger csa of $\g h$ which is impossible. We see that either
$Z\in\g h_+$ or $Z\in\g h_-$.
We call $\g h_0$ in this case a {\em symplectic extrinsic Cartan subalgebra} associated to the data $(\alg, \tialg, \sigma, \theta, \omega)$.

Note that, as $\tih_0 \subset \g h_0$, the restriction of a weight $\omega$ from $\alg$ (w.r.t. $\g h_0$) to $\tih_0$ is a weight from $\tialg$ (w.r.t. $\tih_0$).

For the given $\g h_0$ let $(\g h_0, \Delta, \alpha_0, r)$ be a defining quadruple of the simple symplectic triple $(\g g, \theta, \omega)$. Let $(V,\rho)$ be the representation $\g g$ with highest weight being the fundamental weight $\omega_0$ associated to $\alpha_0$. By restriction of $(V,\rho)$, $(V,\rho\,\vrule\,_{\tialg})$ is a representation of $\tialg$. It is irreducible(why?).
To make sure that $\omega_0$ is invariant under the adjoint action of $\tih$ we will need that $\tilde\omega_0 := \omega_0\,\vrule\,_{\tih_0}$ is a multiple of the fundamental weight (w.r.t. $\tih_0$) associated to ${\tilde\alpha_0}$ (see \cite{NT}).

\begin{prop}
Let $(\alg, \tialg, \sigma, \theta, \omega)$ be a complex semi-simple symplectic extrinsic symmetric quintuple, with symplectic extrinsic csa $\g h_0$ and let $(\g h_0, \Delta, \alpha_0, r)$ (resp. $(\tialg \cap \g h_0, \tilde\Delta, \tilde\alpha_0, \tilde r)$) be a defining quadruple of the simple symplectic triple $(\g g, \sigma, \omega)$ (resp. $(\tialg, \theta, \tilde\omega)$). Let $\omega_0$ be the fundamental weight associated to $(\Delta,\alpha_0)$ and let $(V,\rho)$ be a representation of $\alg$ with highest weight $\omega_0$. Let $\tilde\omega_0$ be $\omega_0\,\vrule\,_{\tialg \cap \g h_0}$.

Let $S:=\{\tilde\omega_i\}$ denote the set of weights of the irreducible representation with highest weight ${\tilde\omega_0}$, and $S_j:=\{\tilde\omega\in S \,|\,  (\tilde\omega-\tilde\omega_0+j\tilde\alpha_0,\tilde\omega_0)=0\}$ for all $j$.

\begin{enumerate}[(i)]

\item If $\alg = \g p \oplus \g h$ is the Cartan decomposition corresponding to $\theta$,
and $V_+ := \bigoplus_{\tilde\omega_i\in \bigcup_n S_{2n}} V_{\tilde\omega_i}$ and $V_- := \bigoplus_{\tilde\omega_i\in \bigcup_n S_{2n+1}} V_{\tilde\omega_i}$

then $\rho(\g h) V_\pm \subset V_\pm$ and $\rho(\g p) V_\pm \subset V_\mp$.

\item We have that $S_0=\{\tilde\omega_0\}$ and $S=S_0\cup S_1\cup S_2$.
\end{enumerate}
\end{prop}

\begin{proof}
The first assertion is more or less clear due to the fact that the action of $\g p$ exchanges weights from $\bigcup_n S_{2n}$ and $\bigcup_n S_{2n+1}$ and that the action of $\g h$ permutes weights inside $\bigcup_n S_{2n}$ and $\bigcup_n S_{2n+1}$.

That $S_0=\{\tilde\omega_0\}$ has been observed in \cite{NT}.

We have that $\g p_-=V_{\tilde\omega_0}^*\otimes V_-$.
We will prove that $S_3$ is empty, and it follows from classical weight-theory that then $S_j$ is empty for all $j\ge 3$.
Suppose that $S_3$ contains a weight $\tilde\omega$. As $\tip := \langle X + \lambda(X) \,|\, X \in \g p_-\rangle$, there is a root $\tilde\alpha$ from $\tip$ such that $\tilde\omega=\tilde\omega_0 +\tilde\alpha$. But $\tilde\alpha$ contains the summand $\tilde\alpha_0$ only $\pm 1$ times, a contradiction with $\tilde\omega\in S_3$.
\end{proof}
\begin{rem}
Note that $V_1$ identifies to the tangent space $T_o \tilde M$ and $V_2$ identifies to the normal space $\mathcal{N}_o \tilde M$.
Another way of describing the situation is then to say $V=\complexes \oplus T_o \tilde M \oplus \mathcal{N}_o \tilde M$ as $\tilde{\g h}$-modules.
\end{rem}

\begin{theo}\label{big_algebra}
Let $(\g s, \sigma, \omega)$ be a simple symplectic triple with isotropy algebra $\g h$, and let $(\g h_0, \Delta, \alpha_0, r)$ be a defining quadruple of $(\g s, \sigma, \omega)$.

We can decompose $\g h$ into $\complexes \oplus \tilde{\g g}$ where $\tilde{\g g}=\bigoplus_i \tilde{\g g}_i$ with $\tilde{\g g}_i$ simple and simple root system $\tilde{\Delta}_i$.
Let $\tilde{\g h}_0$ be the csa of $\tilde{\g g}$ obtained by intersecting $\g h_0$ with $\tilde{\g g}$.

Let $\{{\alpha}_0^i\}$ be the subset of $\Delta$ containing the roots not orthogonal to $\alpha_0$ such that
the restriction $\tilde{\alpha}_0^i$ of ${\alpha}_0^i$ to $\tilde{\g h}_0$ is in $\tilde{\Delta}_i$.
Each $(\tilde{\g g}_i \cap \tilde{\g h}_0, \tilde{\Delta}_i, \tilde{\alpha}_0^i, r_i)$ defines then a symplectic triple
$(\tilde{\g g}_i,\sigma_i,\omega_i)$.

Let $V$ the $\tilde{\g g}$-module with highest weight $\tilde{\alpha}_0 := \alpha_0\,\vrule\,_{\tilde{\g h}_0}$.
Then $V$ decomposes into $V_0 \oplus V_1 \oplus V_2$ as a $\tilde{\g h}$-module
where the weights of $V_j$ are
exactly the restrictions of the roots $\alpha$ of $\g s$ to $\tilde{\g h}_0$ such that $\alpha$ written as $\sum_k n_{\alpha_k} \alpha_k$ with $\alpha_k\in\Delta$ verifies $n_{\tilde{\alpha}_0}=1$ and $\sum_i n_{\tilde{\alpha}_0^i}=j$.

Let $\g g$ be $\g s \g l (V)$, $\sigma$ be the involution $\Ad(I_{V_0,V_1\oplus V_2})$, $\theta$ be the involution $\Ad(I_{V_0\oplus V_1,V_2})$. Identify $\tilde{\g g}$ with its representing subalgebra in $\g s \g l (V)$.
Then $(\alg, \tialg, \sigma, \theta)$ is a semisimple symplectic extrinsic quadruple.

Reciprocally every undecomposable semisimple symplectic extrinsic quadruple can be obtained this way.
\end{theo}

\begin{proof}
The proof of the direct sense goes by reasoning as follows:
Consider the Dynkin diagram of the transvection algebra $\g s$ of a complex simple symplectic symmetric triple together with the distinguished root $\alpha_0$. Consider $\tilde{\alpha}_0^i$ the vertices connected to $\alpha_0$ in the diagram. By removing $\alpha_0$ and all incident edges one obtains one or two connected diagrams, each one with a distinguished vertex $\tilde{\alpha}_0^i$.
One can verify that these diagrams correspond to simple symplectic symmetric triples $(\tialg_i,\sigma_i,\omega_i)$ together with their canonical symplectic form.
Define $V$ as in the theorem. One can then check in case per case that the representation $V$ is exactly the one given in the list from section \ref{complex_classif}.
Reciprocally every example of the list appears in this way, which shows the reverse sense.

Note that strictly speaking we should identify $\tialg$ with a subalgebra of $\g s \g l(V \oplus V^*)$,
but then $\tialg$ is contained in $$\left\{\left(\begin{matrix} A & 0\\ 0 & -^t{A}\end{matrix}\right) \, | \, A \in \g s \g l(V)\right\}$$
which identifies to $\g s \g l(V)$ as a Lie algebra.

\end{proof}

%

\section{Real forms}
\subsection{Definitions}
We say that the symmetric pair over $\reels$ $(\alg_0, \sigma_0)$  is a {\em real form} of the symmetric pair over $\complexes$ $(\alg, \sigma)$ if
$\alg = \alg_0 \otimes_\reels \complexes$, $\sigma$ is the complexification of $\sigma_0$.

We say that the extrinsic symmetric quadruple over $\reels$ $(\alg_0, \tialg_0, \sigma_0, \theta_0)$ is a {\em real form} of the (symplectic) extrinsic symmetric quadruple over $\complexes$ $(\alg, \tialg, \sigma, \theta)$ if
$(\alg_0,\sigma_0)$  is a real form of $(\alg,\sigma)$ and
$(\tialg_0, \theta_0)$ is a real form of $(\tialg, \theta)$.

We say that the real form $(\alg_0, \tialg_0, \sigma_0, \theta_0)$
is {\em split} (resp. {\em compact}), if $\alg_0$ is a split (resp. compact) real form of $\alg$ and $\tialg_0$ is a split (resp. compact) real form of $\tialg$.
The symplectic symmetric triple $(\alg_0, \sigma_0, \omega_0)$, (resp. the extrinsic symplectic symmetric quintuple $(\alg_0, \tialg_0, \sigma_0, \theta_0, \omega_0)$) over $\reels$
is said to be a {\em real form} of the symplectic symmetric triple $(\alg, \sigma,\omega)$, (resp. the extrinsic symmetric quintuple $(\alg, \tialg, \sigma, \theta, \omega)$) over $\complexes$, if $(\alg_0, \sigma_0)$ (resp. $(\alg_0, \tialg_0, \sigma_0, \theta_0)$) is a real form of $(\alg, \sigma)$ (resp. $(\alg, \tialg, \sigma, \theta)$) and $\omega$ is the complexification of $\omega_0$.

Note that $(\alg_0, \tialg_0, \sigma_0, \theta_0)$ is a real form of $(\alg, \tialg, \sigma, \theta)$ precisely if
there exists a complex antilinear involution $\tau$ on $\alg$, commuting with $\sigma$ and $\theta$, verifying $\tau(\tialg)=\tialg$, fixing precisely $\alg_0$ in $\alg$ and $\tialg_0$ in $\tialg$, and if we have $\sigma_0=\sigma\,\vrule\,_{\alg_0}$ and $\theta_0=\theta\,\vrule\,_{\alg_0}$.
Note also that for every real form $\alg_0$ of $\alg$ there is a {\em Cartan involution} $\tau$ which is the unique involution such that if $\alg_0=\g k \oplus \g p$ with $\g k$ and $\g p$ being the $+1$- (resp. $-1$-) eigenspaces of $\tau$, then $\g k \oplus i\g p$ is the compact real form.

We say that the real form $(\alg_0, \tialg_0, \sigma_0, \theta_0)$
is {\em outer} (resp. {\em inner}), if the Cartan involution $\tau$ defining the real form is outer (resp. inner).

\subsection{From complex to compact semi-simple extrinsic symmetric quadruples}

\begin{theo}\label{complextorealform}
Let $(\alg, \tialg, \sigma, \theta)$ be a semisimple extrinsic symmetric quadruple over the field $\corps=\complexes$ such that
the rank of $\tialg$ equals the rank of the isotropy subalgebra $\tih$ of $\tialg$.
Then  $(\alg, \tialg, \sigma, \theta)$ admits a split (resp. compact) real form.
\end{theo}

\begin{proof}
By theorem~\ref{csa} consider a $\sigma$- and $\theta$-invariant csa $\g h_0$ of $\alg$, containing the csa $\tih_0$ of $\tialg$.
There is a natural decomposition $\g h_0=\g h_{0,\reels} \oplus i \g h_{0,\reels}$ as the roots of $\alg$ are contained in a real subspace of $\g h_0^*$.
This decomposition defines an antilinear involution $\tau$ on $\g h_0$.
Note that $\tau$ and $\sigma$ (resp. $\tau$ and $\theta$) commute on $\g h_0$. This follows from the fact that $\sigma$ (resp. $\theta$) maps roots on roots.

To extend $\tau$ to $\alg$ note that the root spaces $\alg_\alpha$ ($\alpha\in\Phi$) are $\tau$-invariant. Consequently it is enough to define $\tau$ on the root spaces $\alg_\alpha$ of a system of simple roots $\Delta\subset\Phi$ such that $\sigma$ and $\tau$ commute, and $\theta$ and $\tau$ commute.
Reciprocally any such choice gives together with the definition of $\tau$ on $\g h_0$ an antilinear involution on $\alg$ for which $\tau(\g h_0)=\g h_0$ and commuting with $\sigma$ and $\theta$.

Start by taking any definition of $\tau$ on $\alg_\alpha$, $\alpha\in\Delta$. As in (Berger \cite{Be} pp. 100-101, Murakami) we can write: $\sigma=\alpha\beta$ with $\tau\alpha=\alpha\tau$ and $\tau\beta=\beta^{-1}\tau$. (We have $\beta=e^{\ad(X)}$, for a unique $X$ such that $\tau(X)=-X$.)

For $X\in\g h_0$, $\alpha\beta\tau(X)=\sigma\tau(X)=\tau\sigma(X)=\alpha\tau\beta(X)=\alpha\beta^{-1}\tau(X)$.
>From this follows $\beta(X)=\beta^{-1}(X)$ for $X\in\g h_0$.
As a consequence $\beta^2\,\vrule\,_{\g h_0}=\Id_{\g h_0}$.
$\beta^2$ stabilizes the root spaces and has to be of the form $\beta^2=e^{\ad(2X_0)}$ with $X_0\in\g h_0$. By uniqueness(Murakami) of the decomposition, we have $\beta=e^{\ad(X_0)}$.
Define $\tau':=e^{\ad(-\frac{1}{2}X_0)} \tau e^{\ad(\frac{1}{2}X_0)}$.
Then (\cite{Be}, Lemma 10.2) $\tau'$ commutes with $\sigma$ and equals $\tau$ on $\g h_0$.

In the next step we are going to look for  $\tau''$ commuting with $\sigma$ and $\theta$ and equal to $\tau$ on $\g h_0$.
Let $\theta=\alpha'\beta'$, with
$\tau'\alpha'=\alpha'\tau'$ and $\tau'\beta'=\beta'^{-1}\tau'$. (We have $\beta'=e^{\ad(X_0')}$, with $\tau'(X_0')=-X_0'$. As before $X_0'\in\g h_0$.)
As in (\cite{Be}, Lemma 10.2) $\alpha'=\beta'^{\frac{1}{2}}\theta\beta'^{-\frac{1}{2}}$.
As a consequence $\alpha'^2=\Id$.
By uniqueness of the decomposition $\sigma\theta=(\sigma \alpha') \beta'$, with respect to $\tau'$, we obtain similarly: $(\sigma \alpha')^2=\Id$, from which it is easy to deduce that $\sigma$ and $\alpha'$ commute. As $\sigma$ and $\theta$ commute, we also deduce that $\sigma$ and $\beta'$ commute. Finally $\sigma$ and $\beta'^{\frac{1}{2}}=e^{\ad({\frac{1}{2}X_0'})}$ commute (see Lemma 10.1 from \cite{Be}).
All in all we obtain that $\tau'':=e^{\ad(-\frac{1}{2}X_0')} \tau e^{\ad(\frac{1}{2}X_0')}$ commutes with $\sigma$ (as $\tau$ and $\beta'^{\frac{1}{2}}$ does). $\tau''$ commutes with $\theta$.

The antiinvolution $\tau''$ defines a split real form of $\alg$. It is easy to see that the automorphism $\tau''$ stabilizes $\tih_0$. The root spaces $\tialg_{\tilde{\alpha}}$ are also $\tau''$-invariant, which implies that $\tau''(\tialg)=\tialg$ and that we obtain a split real form of $(\alg, \tialg, \sigma, \theta)$.

Let us look now for compact real forms.
There is an involution $\iota$ of $\alg$ such that $\iota\,\vrule\,_{\g h_0}=-\Id_{\g h_0}$.
One can choose now a system of non vanishing vectors $X_\alpha\in\alg_\alpha$, $\alpha\in\Phi$, such that $\tau(X_\alpha)=X_\alpha$, $B(X_\alpha,X_{-\alpha})=1$ and $\iota(X_\alpha)=X_{-\alpha}$ (see \cite{HN}).
The vectors $X_\alpha-X_{-\alpha}$, $i(X_\alpha+X_{-\alpha})$ together with $i \g h_{0,\reels}$
define a compact real form of $(\alg, \tialg, \sigma, \theta)$.

\end{proof}

\begin{corr}
Any semisimple extrinsic symmetric quadruple $(\alg, \tialg, \sigma, \theta)$ over the field $\corps=\reels$ such that the rank of $\alg$ and $\tialg$ coincide can be deduced from
the compact real form $(\alg_u, \tialg_u, \sigma, \theta)$ of its complexification by an involution $\tau$ of $\alg_u$ commuting with $\sigma$ and $\theta$ and such that $\tau(\tialg_u)=\tialg_u$ in the following sense:
$\alg_u=\g k \oplus \g p$ and $\alg=\g k \oplus i\g p$ where $\g k$ (resp. $\g p$) is the $+1$- (resp. $-1$-) eigenspace of $\tau$ and
$\tialg_u=\tilde{\g k} \oplus \tilde{\g p}$ and $\tialg=\tilde{\g k} \oplus i\tilde{\g p}$ where $\tilde{\g k}$ (resp. $\tilde{\g p}$) is the $+1$- (resp. $-1$-) eigenspace of $\tau\,\vrule\,_{\tialg_u}$.\footnote{In this setting the condition $\tau(\tialg_u)=\tialg_u$ can be replaced by the condition that $\lambda$ and $\tau$ commute.}
\end{corr}

Note that the involution $\tau$ extends to a complex-linear involution on the complexification of $\alg$. This is what we mean in the following by an involution $\tau$ defining a real form.

\subsection{The symplectic case}

Note that any involution of the outer algebra $\g s \g l(V)$ is either conjugated to the inner involution $\Ad(I_{V_+,V_-})$, where $V=V_+\oplus V_-$, or conjugated to the outer involution $X\mapsto -^{t}X$.

\subsubsection{Definition}
We now turn to the characterization of extrinsic symplectic symmetric spaces as in theorem \ref{big_algebra}.
Let $(\g s, \sigma, \omega)$ be a complex simple symplectic symmetric triple with defining quadruple $(\g h_0, \Delta, \alpha_0, r)$ and $r=1$.

Let $(\g s_0, \sigma_0, \omega_0)$ a real form of $(\g s, \sigma, \omega)$ and let $\tau$ be an involution of $\g s$ commuting with $\sigma$ and
corresponding to the real form $\g s_0$. Furthermore we suppose that $\g h_0$ is stabilized by $\tau$.
Let $h_{\alpha_0}$ be defined as usually. We have to distinguish two cases as $\tau(h_{\alpha_0})=\pm h_{\alpha_0}$.

We say that $(\g s_0, \sigma_0, \omega_0)$ is an {\em extrinsic real form of compact type} (resp. {\em non-compact}) {\em type} of $(\g s, \sigma, \omega)$ if $\tau$ fixes a non vanishing root vector $X_{\alpha_0}$ (resp. if $\tau(X_{\alpha_0})$ is of weight $-\alpha_0$).\footnote{We ignore for the moment the case $\tau(X_{\alpha_0})=-X_{\alpha_0}$.} Note that up to isomorphism, the choice of the root ${\alpha_0}$
does not change the extrinsic real form.
One can "flip" between extrinsic real forms of compact and non-compact type by considering an involution $\tau_0$ that restricts to $-\Id$ on $\g h_0$ and commutes with $\tau$. The bijective correspondence mapping $\tau$ to $\tau\tau_0$ does not give though a bijective correspondence between real forms.

\subsubsection{The correspondence between the two descriptions}

\begin{prop}
There is a bijective correspondence between the extrinsic real forms $\tau_1$ of compact (resp. non-compact) type of complex simple symplectic symmetric triples $(\g s, \sigma, \omega_0)$ (with $\omega_0$ canonical) and the inner (resp. outer) real forms $\tau_2$ of its corresponding extrinsic symplectic symmetric quintuples $(\g s \g l(V),\tilde{\g g},\sigma,\theta, \omega_1)$ (with $\omega_1$ canonical)
such that the restrictions of $\tau_1$ and $\tau_2$ to $\tilde{\g g}$ (respectively $V$) coincide.
\end{prop}

\begin{proof}
We show the correspondence first in the compact case.
Let $\tau$ be the involution corresponding to a compact real form of a complex simple symplectic symmetric triple $(\g s, \sigma, \omega_0)$ and
let be a corresponding defining quadruple.
We have $\g s=(\tialg \oplus \complexes) \oplus (V \oplus V^*)$ with $\tau(\tialg)=\tialg$, $\tialg=\tilde{\g p}\oplus \tilde{\g h}$ with $\tilde{\g h}$ being the centralizer of $X_{\alpha_0}$, which is $\tau$-invariant.

We have $\tau(X_{\alpha_0})=X_{\alpha_0}$. For each weight vector $X_\alpha$ for $V$ we have
$X_\alpha=\ad(X_0) \ldots \ad(X_k)X_{\alpha_0}$ for certain $X_0, \ldots, X_k\in\tilde{\g p}$, and so $\tau(X_\alpha)=\ad(\tau(X_0)) \ldots \ad(\tau(X_k))X_{\alpha_0}$.
It is clear then that the matrix $I_{V_+,V_-}$, with $V_\pm=\langle X_\alpha\in V \,|\, \tau(X_\alpha)=\pm X_\alpha \rangle$, is such that the inner involution $\Ad(I_{V_+,V_-})$ defines a real form of $(\g s \g l(V),\tilde{\g g},\sigma,\theta, \omega_1)$.

Reciprocally suppose given an inner real form of $(\g s \g l(V),\tilde{\g g},\sigma,\theta, \omega_1)$ with defining
involution $\Ad(I_{V_+,V_-})$. By restriction to $\tialg$ we obtain a real form of $\tialg$. The involution $\sigma$ fixes the highest weight vector $X_\omega$ of $V$, as a consequence $\tau (X_\omega)=\tau \sigma (X_\omega)=\sigma \tau (X_\omega)$ which means that
$\tau (X_\omega)=\pm X_\omega$. By eventually exchanging $V_+$ and $V_-$ we can suppose $X_\omega\in V_+$.
The highest weight vector is identical to the root vector $X_{\alpha_0}$ in the second way of describing extrinsic spaces.
Again we have for every root vector $X_\alpha\in V$ (resp. we can define for $V_*$):
$X_\alpha=\ad(X_0) \ldots \ad(X_k)X_{\pm\alpha_0}$ for certain $X_0, \ldots, X_k\in\tilde{\g g}$, and so $\tau(X_\alpha)=\ad(\tau(X_0)) \ldots \ad(\tau(X_k))X_{\pm\alpha_0}$, which shows that $\tau$ can be seen as an extrinsic real form
of compact type of the complex simple symplectic symmetric triple $(\g s, \sigma, \omega_0)$.

Suppose now that we are in the non compact case.
Let $\tau$ be the involution corresponding to a real form of non-compact type of a complex simple symplectic symmetric triple $(\g s, \sigma, \omega_0)$ and consider a corresponding defining quadruple. There is an involution $\tau_0$ that commutes with $\tau$ and $\sigma$
and restricts to $-\Id$ on the csa $\g h_0$.
If is clear then that $\tau':=\tau\tau_0$ is a real form of compact type of $(\g s, \sigma, \omega_0)$
which corresponds to an inner real form $\Ad(I_{V_+,V_-})$ of $\g s \g l(V)$.
Finally $X\mapsto \tau_0'\circ\Ad(I_{V_+,V_-})(X)$ defines an outer real form of $(\g s \g l(V),\tilde{\g g},\sigma,\theta, \omega_1)$
corresponding to $\tau$, where $\tau_0'$ is an involution of $\g s \g l(V)$ restricting to $-\Id$ on the csa $\g h_0$ and commuting with $\sigma$, $\Ad(I_{V_+,V_-})$ and $\theta$.

Finally suppose that we have an outer real form of $(\g s \g l(V),\tilde{\g g},\sigma,\theta, \omega_1)$
corresponding to $\tau$. If $\tau_0$ is a mapping that commutes with $\tau$, $\sigma$ and $\theta$
and restricts to $-\Id$ on the csa $\g h_0$, then $\tau':=\tau\tau_0$ is an inner real form of $(\g s \g l(V),\tilde{\g g},\sigma,\theta, \omega_1)$ and corresponds to a real form $\tau''$ of compact type of $(\g s, \sigma, \omega_0)$. There is an involution $\tau_0'$ of $\g s$ restricting to $-\Id$ on the csa $\g h_0$ and commuting with $\sigma$ and $\tau''$. The involution $\tau_2:=\tau''\tau_0'$ obviously answers the requirements.

\end{proof}

\pagebreak

\section{List of semisimple symplectic extrinsic symmetric spaces}\label{list}

\subsection{The complex classification}\label{complex_classif}
From the list of H. Nakagawa and R. Takagi we obtain the following classification of
the full\footnote{We only consider extrinsic subspaces not contained in a proper totally geodesic subspace $\g s \g l(N_0+1,\complexes)/\g g\g l (N_0,\complexes)$ of $M$.}
undecomposable non totally geodesic semisimple complex symplectic extrinsic symmetric quintuples $(\alg,\tialg,\sigma,\theta,\omega)$. If $V$ denotes the first fundamental representation of $\rho: \alg \to \g g \g l(V)$, we consider the restriction of $\rho$ to $\tialg$, which we also denote by $V$ and which we will list in the following table.
The outer symmetric pair is $\g s \g l(N+1,\complexes)/\g g\g l (N,\complexes)$.
\[\begin{array}{|l|l|l|l|}\hline
\tialg / \tih & V & N & \mbox{conditions}\\\hline
\g s \g l(n+1,\complexes)/\g g \g l(n,\complexes) & S^2(\complexes^{n+1}) & \frac{1}{2}(n^2+3n) & n \ge 1\\
\g s \g l(n+1,\complexes)/\g s \g l(2,\complexes)\oplus\g s \g l(n-1,\complexes)\oplus\complexes & \bigwedge^2 (\complexes^{n+1}) & \frac{1}{2}(n+1)n-1 & n \ge 4\\
\g s \g o(n+1,\complexes)/\g s \g o(2,\complexes)\oplus \g s \g o(n-1,\complexes) & \complexes^{n+1} & n & n \ge 4\\
\g s \g o(10,\complexes)/\g g \g l(5,\complexes) & \mbox{halfspin rep.} & 15 & \\
\g e_6^\complexes/\g s \g o(10,\complexes)\oplus \complexes & \complexes^{27} & 26 &\\
\g s \g l(a+1,\complexes)/\g g \g l(a,\complexes) \times \g s \g l(b+1,\complexes)/\g g \g l(b,\complexes)& \complexes^{a+1}\otimes  \complexes^{b+1} & ab+a+b & 1 \le a \le b\\
\hline
\end{array}\]

\subsection{Real forms}\label{realforms}
We first prove the following result:

\begin{prop}
Let $(\g s \g l(V),\tialg,\sigma,\theta)$ be an example of the list in section \ref{complex_classif} and $\phi$ be an involution commuting with $\sigma$ and $\theta$. If $\phi$ fixes $\tialg$ then $\phi$ is the identity on $\g s \g l(V)$.
\end{prop}
\begin{proof}
The proof is based on the observation that every involution of $\g s \g l(V)$ is conjugated either to an $\Ad(I_{p,q})$ commuting with $\sigma$ or to a transposition commuting with $\sigma$.
As the representation defining $\tialg$ is irreducible, $\Ad(I_{p,q})$ fixes $\tialg$ iff $\Ad(I_{p,q})$ is the identity.
If $\tialg$ is fixed by a transposition, we have $\tialg\subset\g s\g o(V)$. By examining the corresponding cases from the list, see section \ref{realforms}, we see that as least a generator of the center of the isotropy algebra of $\tialg$ is not fixed by the transposition.
\end{proof}

\begin{corr}
Any real form of $(\g s \g l(V),\tialg,\sigma,\theta)$, where  $(\g s \g l(V),\tialg,\sigma,\theta)$ is in the list of section \ref{complex_classif}, is determined by its restriction to $(\tialg,\theta)$.
\end{corr}
\begin{proof}
Let $\tau$ and $\phi$ be involutions determining real forms of $(\g s \g l(V),\tialg,\sigma,\theta)$ such that $\tau\,\vrule\,_{\tialg}=\phi\,\vrule\,_{\tialg}$.
If $\tau\phi^{-1}$ is an involution, we are done by the preceding theorem. This is always the case if $\tau$ and $\phi$ commute.

We can apply a Berger-type argument and find an inner automorphism $\psi$ fixing $\tialg$ and such that $\psi \phi \psi^{-1}$ commutes to $\tau$.
We can then conclude by the preceding.

The argument is as follows: One can write $\phi=\alpha \beta$ with $\alpha$ an involution commuting to $\tau$ and $\beta=e^{\ad X}$ with $\tau(X)=-X$. For $A\in\tialg$ we have $\tau\alpha(A)=\alpha \tau(A)$ which shows that $\tau$ fixes $\alpha(A)$.
We also have $A=\phi(A)=\alpha \beta(A)$ giving $\alpha(A)=\beta(A)$, from which we obtain $\alpha(A)=\tau\alpha(A)=\tau\beta(A)=\beta^{-1}\tau(A)=\beta^{-1}(A)$ and then $\beta(A)=\beta^{-1}(A)$.
As $e^{\ad tX}=\Id$ implies $tX=0$, we can write $\beta^{\frac{1}{2}}(A)=\beta^{-\frac{1}{2}}(A)$ and then $\beta(A)=A$ and $\alpha(A)=A$.
It is known that one can take $\psi=\beta^{\frac{1}{2}}$ and that for this choice $\psi \phi \psi^{-1}=\alpha$.
\end{proof}

We use here the classification of semi-simple symmetric spaces from \cite{Be}.

\begin{theo}\label{complextorealform}
The real forms of the symplectic extrinsic semi-simple quintuples from section \ref{complex_classif} are:

\centerline{
$\begin{array}{|l|l|l|l|}\hline
\tialg / \tih & \alg / \g h & N & \mbox{Remarks}\\\hline
\frac{\g s \g u(n+1-q, q)}{\g u(n-q,q)} & \frac{\g s \g u(N+1-Q, Q)}{\g u(N-Q,Q)} & \frac{1}{2}(n^2+3n)&\begin{array}{l}q=0\ldots n, n\ge 1\\Q=q+(n-q)q\end{array}\\
\rule[-3mm]{0pt}{5ex}
\frac{\g s \g l(n+1,\reels)}{\g g \g l(n,\reels)} & \frac{\g s \g l(N+1,\reels)}{\g g \g l(N,\reels)} & & \\
\hline
\frac{\g s \g u(n+1-p,p)}{\g s \g u(2)\oplus\g s \g u(n-1-p,p)\oplus u(1)} & \frac{\g s \g u(N+1-P, P)}{\g u(N-P,P)}& \frac{1}{2}(n+1)n-1 & \begin{array}{l}p=0\ldots n-1, n\ge 4\\P=(n-p+1)p\end{array}\\
\frac{\g s \g u(n+1-q,q)}{\g s \g u(1,1)\oplus\g s \g u(n-q,q-1)\oplus u(1)}&\frac{\g s \g u(Q,N+1-Q)}{\g u(Q-1,N+1-Q)}& &\begin{array}{l}q=1\ldots\lceil\frac{n}{2}\rceil\\Q=n+(n-q-1)q\end{array}\\
\frac{\g s \g l(n+1,\reels)}{\g s \g l(2,\reels)\oplus\g s \g l(n-1,\reels)\oplus\reels} & \frac{\g s\g l(N+1,\reels)}{\g g \g l(N,\reels)} & & \\
\rule[-3mm]{0pt}{5ex}
\frac{\g s \g l(m+1,\quaternions)}{\g s \g l(1,\quaternions)\oplus\g s \g l(m,\quaternions)\oplus\reels}& \frac{\g s\g l(N+1,\reels)}{\g g \g l(N,\reels)} & & n=2m+1\\
\hline
\rule[-3mm]{0pt}{5ex}
\frac{\g s \g o(n+1-p,p)}{\g s \g o(n-1-p,p)\oplus \g s \g o(2)}&
\frac{\g s \g u(N+1-p,p)}{\g \g u(N-p,p)}& n & p=0\ldots n-1, n\ge 4 \\
\rule[-3mm]{0pt}{5ex}
\frac{\g s \g o(n+1-q,q)}{\g s \g o(n-q,q-1)\oplus \g s \g o(1,1)}&
\frac{\g s \g l(N+1,\reels)}{\g g\g l(N,\reels)} & & q=1\ldots \lceil\frac{n}{2}\rceil\\
\hline
%
\rul\fracl{\g s \g o(10)}{\g u(5)}&
\fracl{\g s \g u(16)}{\g u(15)}& 15 & \\
\rul\fracl{\g s \g o(10-2p,2p)}{\g u(5-p,p)}& \fracl{\g s\g u(8,8)}{\g u(7,8)} & & p=1,2\\
\rul\fracl{\g s \g o^*(10)}{\g u(5-q,q)}& \fracl{\g s\g u(10,6)}{\g u(10,5)}& & q=0,1\\
\rul\fracl{\g s \g o^*(10)}{\g u(3,2)}& \fracl{\g s\g u(10,6)}{\g u(9,6)}& &\\
\rul\fracl{\g s \g o(5,5)}{\g g\g l(5,\reels)}& \fracl{\g s \g l(16,\reels)}{\g g \g l(15,\reels)} & &\\
\hline
\rul\fracl{\g e_6}{\g s\g o(10)\oplus\g s \g o(2)} & \fracl{\g s \g u(27)}{\g u(26)}& 26 & \\
\rul\fracl{\g e^3_6}{\g s\g o(10)\oplus\g s \g o(2)}& \fracl{\g s\g u(11,16)}{\g u(10,16)}& &\\
\rul\fracl{\g e^2_6}{\g s\g o^*(10)\oplus\g s \g o(2)}& \fracl{\g s\g u(12,15)}{\g u(11,15)}& &\\
\rul\fracl{\g e^3_6}{\g s\g o^*(10)\oplus\g s \g o(2)} & \fracl{\g s\g u(16,11)}{\g u(15,11)}& &\\
\rul\fracl{\g e^2_6}{\g s\g o(6,4)\oplus\g s \g o(2)}& \fracl{\g s\g u(15,12)}{\g u(14,12)} & &\\
\rul\fracl{\g e^3_6}{\g s\g o(2,8)\oplus\g s \g o(2)} &
\fracl{\g s\g u(11,16)}{\g u(10,16)}& &\\
\rul\fracl{\g e^1_6}{\g s\g o(5,5)\oplus\reels}
& \fracl{\g s\g l(27,\reels)}{\g g \g l(26,\reels)}& &\\
\rul\fracl{e^4_6}{\g s\g o(1,9)\oplus\reels}
& \fracl{\g s\g l(27,\reels)}{\g g \g l(26,\reels)}& &\\
\hline
\frac{\g s \g u(p_1+1,q_1)}{\g u(p_1,q_1)} \times \frac{\g s \g u(p_2+1,q_2)}{\g u(p_2,q_2)} &
\frac{\g s\g u(P+1,Q)}{\g u(P,Q)} &
ab+a+b& \begin{array}{l}1\le a\le b\\a=p_1+q_1, b=p_2+q_2\\P:=p_1 p_2+p_1+p_2+q_1 q_2\\Q:=p_1 q_2+q_1 p_2\end{array}\\
\frac{\g s \g l(a+1,\reels)}{\g g\g l(a,\reels)} \times \frac{\g s \g l(b+1,\reels)}{\g g\g l(b,\reels)} &
\frac{\g s\g l(N+1,\reels)}{\g g\g l(N,\reels)} &
& \\
\rule[-3mm]{0pt}{5ex}
\frac{\g s \g l(a+1,\complexes)}{\g g\g l(a,\complexes)} &
\frac{\g s \g u(R+1,S)}{\g u(R,S)} &
& R:=\frac{a(a+3)}{2}, S:=\frac{a(a+1)}{2}\\
\hline
\end{array}$
}
\end{theo}

\noindent To prove this we give details of the real forms of the outer and inner spaces and the corresponding involutions.
Note that for the given pairs $(\g g,\tilde{\g g})$ of section \ref{complex_classif} if two involutions (determining real forms) of $\g g$ commuting with $\sigma$ and $\theta$ coincide on $\tilde{\g g}$ they are necessarily equal, so a real form of $\tilde{\g g}/\tilde{\g h}$ determines at most one real form of $\g g/\g h$.

\subsubsection{The outer space}\label{os}
The exterior space $\g s \g l(N+1,\complexes)/\g g \g l(N,\complexes)$ admits the following real forms:
\begin{enumerate}
\item The compact real form $\lie{\complexes P^N} = \g s \g u(N+1)/\g u(N)$.
\item The real forms $\lie{\complexes P^{N-p,p}} = \g s \g u(N+1-p, p)/\g u(N-p,p)$ for $p=1,\ldots,N$, obtained by the inner involution $Ad(I_{N+1-p,p})$, where $I_{P,Q}=\left( \begin{smallmatrix}
I_P & \mathbf{0} \\
\mathbf{0} & -I_Q
\end{smallmatrix} \right)$.
\item The split real form $\g s \g l(N+1,\reels)/\g g \g l(N,\reels)$. It is obtained from the compact real form using the exterior involution $\tau : A \mapsto -^t A$.
\end{enumerate}
Note that for $N+1=2m$, $\g s \g l(2m,\complexes)$ admits an additional real form: $\g s \g l(m,\quaternions)$ but it does not give rise to a real form of the complex outer space as $N$ is odd.

\subsubsection{$\g s \g l(n+1,\complexes)/\g g \g l(n,\complexes)$, $V=S^2(\complexes^{n+1})$ with $n\ge 1$}\label{is}\label{S2} Let $N:=\frac{1}{2}(n^2+3n)$.
If $(e_0, e_1, \ldots, e_n)$ is a basis of weight vectors of the standard representation $\complexes^{n+1}$ of $\g s \g l(n+1,\complexes)$,
we have $V_+ = \complexes (e_0\odot e_0) \oplus \langle e_i \odot e_j \,|\, 1\le i\le j\le n\rangle$ and
$V_- = \langle e_0 \odot e_i \,|\, 1\le i\le n \rangle$.

The real forms of $\g s \g l(n+1,\complexes)/\g g \g l(n,\complexes)$ are:
\begin{enumerate}
\item The compact one $\lie{\complexes P^n} = \g s \g u(n+1)/\g u(n)$.
  It embeds into $\lie{\complexes P^N}$.
\item The real forms $\lie{\complexes P^{n-q,q}} = \g s \g u(n+1-q, q)/\g u(n-q,q)$ for $q=1,\ldots,n$, obtained by the inner involution $Ad(I_{n+1-q,q})$.
    It embeds into $\lie{\complexes P^{N-Q,Q}}$ for $Q=q+(n-q)q$.
\item The split real form $\g s \g l(n+1,\reels)/\g g \g l(n,\reels)$. It corresponds to an outer involution.
It embeds into $\g s \g l(N+1,\reels)/\g g \g l(N,\reels)$.
\end{enumerate}

\subsubsection{$\g s \g l(n+1,\complexes)/\g s \g l(2,\complexes)\oplus\g s \g l(n-1,\complexes)\oplus\complexes$, $V=\bigwedge^2 (\complexes^{n+1})$ with $n\ge 4$} Let $N:=\frac{1}{2}(n+1)n-1$.
Choose a basis of weight vectors $(e_1,\ldots,e_{n+1})$ w.r. to the first fundamental representation $\complexes^{n+1}$ of
$\g s \g l(n+1,\complexes)$. $V$ is then spanned by the weight vectors $e_i \wedge e_j$ for $i<j$. Let $\omega_i + \omega_j$ be the weight of $e_i \wedge e_j$. We have then: $S_0=\{\omega_1+\omega_2\}$, $S_1=\{\omega_1+\omega_j \,|\, j\ge 3\}\cup \{\omega_2+\omega_j \,|\, j\ge 3\}$, $S_2=\{\omega_i+\omega_j \,|\, 3\le i<j\}$.
$|S_0 \cup S_2|=1+\frac{1}{2}(n-1)(n-2)$, $|S_1|=2 (n-1)$.

The real forms of $\g s \g l(n+1,\complexes)/\g s \g l(2,\complexes)\oplus\g s \g l(n-1,\complexes)\oplus\complexes$ are:
\begin{enumerate}
\item The compact one $\lie{G_2(\complexes^{n+1})} = \g s \g u(n+1)/\g s \g u(2)\oplus\g s \g u(n-1)\oplus u(1)$.
  It embeds into $\lie{\complexes P^N}$.

\item The real forms $\g s \g u(n+1-p,p)/\g s \g u(2)\oplus\g s \g u(n-1-p,p)\oplus u(1)$ for $p=1,\ldots,n-1$.
They embed into $\g s \g u(N+1-P, P)/\g u(N-P,P)$, with $P=(n-p+1)p$.
\item The real forms $\g s \g u(n+1-q,q)/\g s \g u(1,1)\oplus\g s \g u(n-q,q-1)\oplus u(1)$ for $q=1,\ldots,\lceil\frac{n}{2}\rceil$.
They embed into $\g s \g u(Q,N+1-Q)/\g u(Q-1,N+1-Q)$, with $Q=n+(n-q-1)p$.
\item The split real form $\g s \g l(n+1,\reels)/\g s \g l(2,\reels)\oplus\g s \g l(n-1,\reels)\oplus\reels$.
It embeds into $\g s\g l(N+1,\reels)/\g g \g l(N,\reels)$.
\item In case $n=2m+1$, the real form $\g s \g l(m+1,\quaternions)/\g s \g l(1,\quaternions)\oplus\g s \g l(m,\quaternions)\oplus\reels$. It embeds into $\g s\g l(N+1,\reels)/\g g \g l(N,\reels)$.
\end{enumerate}

\subsubsection{$\g s \g o(n+1,\complexes)/\g s \g o(n-1,\complexes)\oplus \g s \g o(2,\complexes)$, $V=\complexes^{n+1}$ with $n\ge 4$} Here $N:=n$.

We use the following embedding:
$\g s\g o(n+1,\complexes)\simeq \tilde {\g k} \oplus \tilde {\g p} \subset \g s \g l(n+1,\complexes)$,
where $\tilde {\g k}$ (resp. $\tilde {\g p}$) defined by
$$\begin{array}{l l l} \tilde {\g k} & := &
\lbrace \left( \begin{smallmatrix}
a & 0  & {}^t\mathbf{0} \\
0 & -a & {}^t\mathbf{0} \\
\mathbf{0} & \mathbf{0} & A
\end{smallmatrix} \right) \, | \, a\in\complexes, A\in\g s\g o(n-1, \complexes) \rbrace,\\
\tilde{\g p} & := &
\lbrace \left( \begin{smallmatrix}
0 & 0 & -^t v \\
0 & 0 & -^t u \\
u & v & \mathbf{0}
\end{smallmatrix} \right) \, | \, u,v\in\complexes^{n-1}\rbrace
\end{array}$$
are the $+1$- (resp. $-1$-)eigenspaces of $\theta\,\vrule\,_{\tialg}$.

We obtain the following real forms:
\begin{enumerate}
\item The compact one $\lie{G_2(\reels^{n+1})} = \g s \g o(n+1)/\g s \g o(n-1)\oplus\g s \g o(2)$.
  It embeds into $\lie{\complexes P^N}=\g s \g u(n+1) / \g \g u(n).$

\item The real forms $\g s \g o(n+1-p,p)/\g s \g o(n-1-p,p)\oplus \g s \g o(2)$ for $p=1,\ldots,n-1$. They embed into $\g s \g u(n+1-p,p) / \g \g u(n-p,p)$. The real form corresponds to the involution $\tau=\Ad(I_{n+1-p,p})$.

\item The real forms $\g s \g o(n+1-p,p)/\g s \g o(n-p,p-1)\oplus \g s \g o(1,1)$ for $p=1,\ldots,\lceil\frac{n}{2}\rceil$.
They embed into $\g s \g l(N+1,\reels) / \g g\g l(N,\reels)$.
The real form corresponds to the involution defined by $\tau(X)=\Ad(I_{n+2-p,p-1})(-^tX)$.

\end{enumerate}

\subsubsection{$\g s \g o(10,\complexes)/\g g \g l(5,\complexes)$, $V=\complexes^{16}$}

$V$ is the halfspin representation of $\g s \g o(10,\complexes)$.
It can be described as follows: Let $W$ be the standard representation of $\g s \g o(10,\complexes)$ and $g$ an invariant metric on $W$. Let $W=E \oplus F$ be a decomposition of $W$ into two isotropic $\g g \g l(5,\complexes)$-invariant spaces. $F$ is then dual to $E$ as a $\g g \g l(5,\complexes)$ representation.
Let $\mathbf{e}:=(e_1,\ldots,e_5)$ be a basis of $E$, and let $\mathbf{f}:=(f_1,\ldots,f_5)$ be the basis of $F$ dual to $\mathbf{e}$.
The Cartan subalgebra $\g h_0$ is chosen to be $\{A\in \g g \g l(5,\complexes)\,|\, \forall i, A e_i\in \complexes e_i \}$ so that the $e_i$ (and the $f_i$) are weight vectors. Let $\alpha_i$ be the corresponding weights, {\em i.e.}
$\forall A\in\g h_0, A(e_i)=\alpha_i(A)e_i$. We have $\forall A\in\g h_0, A(f_i)=-\alpha_i(A)f_i$.
The $e_i\wedge e_j$, $f_i\wedge f_j$, for $i<j$, and the $e_i \wedge f_j$ for all $i$ and $j$ are root vectors corresponding to
the weights $\alpha_i + \alpha_j$, $-\alpha_i-\alpha_j$ and $\alpha_i-\alpha_j$.
The halfspin representation $\mathbf{S}$ can be characterized as follows: $\mathbf{S}=\bigwedge^5 E \oplus \bigwedge^3 E \oplus E$. The weights of $\bigwedge^k E$ are precisely the $\frac{1}{2}\sum_i \pm \alpha_i$ where exactly $k$ signs are positive.
With respect to the notations we used before, we have that $S_j$ is the set of weights of $\bigwedge^{5-2j} E$.

We obtain the following real forms:
\begin{enumerate}
\item The compact real form $\g s \g o(10)/\g u(5)$.
  It embeds into $\lie{\complexes P^{15}}$.

\item The real forms $\g s \g o(10-2p,2p)/\g u(5-p,p)$, $p=1,2$.
They embed into $\g s\g u(8,8)/\g u(7,8)$.

\item The real forms $\g s \g o^*(10)/\g u(5-p,p)$, $p=0,1,2$.
They embed into $\g s\g u(10,6)/\g u(10,5)$ in the cases $p=0,1$ and into $\g s\g u(10,6)/\g u(9,6)$ in the case $p=2$.

\item The split real form $\g s \g o(5,5)/\g g\g l(5,\reels)$. It embeds into $\g s \g l(16,\reels)/\g g \g l(15,\reels)$.
\end{enumerate}

\subsubsection{$\g e_6^\complexes/\g s \g o(10,\complexes)\oplus\complexes$, $V=\complexes^{27}$}
Let $\tilde\omega_0$ be the highest weight of the first fundamental representation of $\g e_6^\complexes$ which is of dimension $27$. Its weights, obtained completely in this case by letting act the Weyl group on the highest weight, can be partitioned into the following three sets:
$S_0=\{ \tilde\omega_0\},$
\medskip

$S_1=\{ \tilde\omega_0-\beta \,|\, \beta$ in the list\footnote{Here $\begin{smallmatrix}& &n_5& &\\n_0&n_1&n_2&n_3&n_4\end{smallmatrix}$ stands for the sum of roots $\sum_i n_i \alpha_i$, where the $\alpha_i$ are simple roots corresponding to the diagram:
$\begin{xy}
  \xymatrix@=6pt{           &                  & \alpha_5\ar@{-}[d]     &                     &  \\
           \alpha_0\ar@{-}[r] &  \alpha_1\ar@{-}[r] & \alpha_2\ar@{-}[r]     &    \alpha_3\ar@{-}[r]  & \alpha_4
}
\end{xy}.$}
$$\begin{matrix}
\begin{smallmatrix}& &0& &\\1&0&0&0&0\end{smallmatrix}, &
\begin{smallmatrix}& &0& &\\1&1&0&0&0\end{smallmatrix}, &
\begin{smallmatrix}& &0& &\\1&1&1&0&0\end{smallmatrix}, &
\begin{smallmatrix}& &0& &\\1&1&1&1&0\end{smallmatrix}, \\ \\
\begin{smallmatrix}& &0& &\\1&1&1&1&1\end{smallmatrix}, &
\begin{smallmatrix}& &1& &\\1&1&1&0&0\end{smallmatrix}, &
\begin{smallmatrix}& &1& &\\1&1&1&1&0\end{smallmatrix}, &
\begin{smallmatrix}& &1& &\\1&1&1&1&1\end{smallmatrix}, \\ \\
\begin{smallmatrix}& &1& &\\1&1&2&1&0\end{smallmatrix}, &
\begin{smallmatrix}& &1& &\\1&2&2&1&0\end{smallmatrix}, &
\begin{smallmatrix}& &1& &\\1&1&2&1&1\end{smallmatrix}, &
\begin{smallmatrix}& &1& &\\1&1&2&2&1\end{smallmatrix}, \\ \\
\begin{smallmatrix}& &1& &\\1&2&2&1&1\end{smallmatrix}, &
\begin{smallmatrix}& &1& &\\1&2&2&2&1\end{smallmatrix}, &
\begin{smallmatrix}& &1& &\\1&2&3&2&1\end{smallmatrix}, &
\begin{smallmatrix}& &2& &\\1&2&3&2&1\end{smallmatrix}\},\\
\end{matrix}$$
\medskip

$S_2=\{ \tilde\omega_0-\beta \,|\, \beta$ in the list
$$\begin{matrix}
\begin{smallmatrix}& &1& &\\2&2&2&1&0\end{smallmatrix}, &
\begin{smallmatrix}& &1& &\\2&2&2&1&1\end{smallmatrix}, &
\begin{smallmatrix}& &1& &\\2&2&2&2&1\end{smallmatrix}, &
\begin{smallmatrix}& &1& &\\2&2&3&2&1\end{smallmatrix}, &
\begin{smallmatrix}& &1& &\\2&3&3&2&1\end{smallmatrix}, \\ \\
\begin{smallmatrix}& &2& &\\2&2&3&2&1\end{smallmatrix}, &
\begin{smallmatrix}& &2& &\\2&3&3&2&1\end{smallmatrix}, &
\begin{smallmatrix}& &2& &\\2&3&4&2&1\end{smallmatrix}, &
\begin{smallmatrix}& &2& &\\2&3&4&3&1\end{smallmatrix}, &
\begin{smallmatrix}& &2& &\\2&3&4&3&2\end{smallmatrix}\}.\\
\end{matrix}$$
\medskip

Note that the weights of $S_1$ (resp. $S_2$), restricted to $\tilde\omega_0^\perp\cap \tih_0$, are the weights of the halfspin (resp. standard) representation of $\g s \g o (10,\complexes)$.

We obtain the following real forms of $\g e_6^\complexes/\g s \g o(10,\complexes)\oplus\complexes$:

\begin{enumerate}
\item The ones corresponding to an inner involution
\[
{\begin{array}{|l|c|l|}\hline\rule[-2mm]{0pt}{4ex}
\mbox{real form of } \tialg/\tih & \mbox{corresponding diagram} & \mbox{embeds into}\\\hline
\rule[-4mm]{0pt}{6ex}\g e_6/\g s\g o(10)\oplus\g s \g o(2) &
\raisebox{+12pt}{\begin{xy}
  \xymatrix@=6pt{           &                  & \cdot\ar@{-}[d]     &                     &  \\
           \bullet\ar@{-}[r] &  \cdot\ar@{-}[r] & \cdot\ar@{-}[r]     &    \cdot\ar@{-}[r]  & \cdot
}
\end{xy}}
& \lie{\complexes P^{26}}\\
\hline
\rule[-4mm]{0pt}{6ex}\g e^3_6/\g s\g o(10)\oplus\g s \g o(2) &
\raisebox{+12pt}{\begin{xy}
  \xymatrix@=6pt{           &                  & \cdot\ar@{-}[d]     &                     &  \\
            \bullet\drop\cir<3.5pt>{}\ar@{-}[r] &  \cdot\ar@{-}[r] & \cdot\ar@{-}[r]     &    \cdot\ar@{-}[r]  & \cdot
}
\end{xy}}
& \g s\g u(11,16)/\g u(10,16)\\
\hline
\rule[-4mm]{0pt}{6ex}\g e^2_6/\g s\g o^*(10)\oplus\g s \g o(2) &
\raisebox{+12pt}{\begin{xy}
  \xymatrix@=6pt{           &                    & \circ\ar@{-}[d]     &                     &  \\
            \bullet\ar@{-}[r] &  \cdot\ar@{-}[r] & \cdot\ar@{-}[r]     &    \cdot\ar@{-}[r]  & \cdot
}
\end{xy}}
& \g s\g u(12,15)/\g u(11,15)\\
\hline
\rule[-4mm]{0pt}{6ex}\g e^3_6/\g s\g o^*(10)\oplus\g s \g o(2) &
\raisebox{+12pt}{\begin{xy}
  \xymatrix@=6pt{           &                    & \circ\ar@{-}[d]     &                     &  \\
            \bullet\drop\cir<3.5pt>{}\ar@{-}[r] &  \cdot\ar@{-}[r] & \cdot\ar@{-}[r]     &    \cdot\ar@{-}[r]  & \cdot
}
\end{xy}}
 & \g s\g u(16,11)/\g u(15,11)\\
\hline
\rule[-4mm]{0pt}{6ex}\g e^2_6/\g s\g o(6,4)\oplus\g s \g o(2)&
\raisebox{+12pt}{\begin{xy}
  \xymatrix@=6pt{           &                    & \cdot\ar@{-}[d]     &                     &  \\
            \bullet\ar@{-}[r] &  \cdot\ar@{-}[r] & \cdot\ar@{-}[r]     &    \circ\ar@{-}[r]  & \cdot
}
\end{xy}}
 & \g s\g u(15,12)/\g u(14,12)\\
\hline
\rule[-4mm]{0pt}{6ex}\g e^3_6/\g s\g o(2,8)\oplus\g s \g o(2) &
\raisebox{+12pt}{\begin{xy}
  \xymatrix@=6pt{           &                    & \cdot\ar@{-}[d]     &                     &  \\
            \bullet\ar@{-}[r] &  \cdot\ar@{-}[r] & \cdot\ar@{-}[r]     &    \cdot\ar@{-}[r]  & \circ
}
\end{xy}}
& \g s\g u(11,16)/\g u(10,16)\\
\hline
\end{array}}
\]
In the diagrams the signs $\bullet$ define the root vectors $X_{\alpha_i}$ such that $\theta(X_{\alpha_i})=-X_{\alpha_i}$ and the signs $\circ$ define the root vectors $X_{\alpha_i}$ such that $\tau(X_{\alpha_i})=-X_{\alpha_i}$. The sign $\obullet$ stands for both $\circ$ and $\bullet$.

\noindent The first one in the list is the compact real form.

\item The ones corresponding to an outer involution
\[
{\begin{array}{|l|c|l|}\hline\rule[-2mm]{0pt}{4ex}
\mbox{real form of } \tialg/\tih & \mbox{embeds into}\\\hline
\rule[-2.5mm]{0pt}{5ex}\g e^1_6/\g s\g o(5,5)\oplus\reels
& \g s\g l(27,\reels)/\g g \g l(26,\reels)\\
\hline
\rule[-2.5mm]{0pt}{5ex}e^4_6/\g s\g o(1,9)\oplus\reels
& \g s\g l(27,\reels)/\g g \g l(26,\reels)\\
\hline
\end{array}}\]

\noindent The first one in the list is the split real form.
\end{enumerate}

\subsubsection{$\g s \g l(a+1,\complexes)/\g g \g l(a,\complexes) \times \g s \g l(b+1,\complexes)/\g g \g l(b,\complexes)$, $V= \complexes^{a+1}\otimes\complexes^{b+1}$ with $1 \le a \le b$}\label{eol} Let $N:=ab+a+b$.

\begin{enumerate}
\item The compact real form $\g s \g u(a+1)/\g u(a) \times \g s \g u(b+1)/\g u(b)$. It embeds into $\lie{\complexes P^{N}}$.
\item The real form $\g s \g u(p_1+1,q_1)/\g u(p_1,q_1) \times \g s \g u(p_2+1,q_2)/\g u(p_2,q_2)$, with $p_1+q_1=a$, $p_2+q_2=b$, $q_1+q_2\ge 1$. It embeds into $\lie{\complexes P^{P,Q}}=\g s\g u(P+1,Q)/\g u(P,Q)$, with $P:=p_1 p_2+p_1+p_2+q_1 q_2$, $Q:=p_1 q_2+q_1 p_2$. The corresponding involution is $\Ad(I_{p_1+1,q_1}\otimes I_{p_2+1,q_2})$.

\item The real form of the inner space $\g s \g l(a+1,\reels)/\g g\g l(a,\reels) \times \g s \g u(b+1)/\g u(b)$. Suppose that there is a real form of the outer space giving by restriction the given real form of the inner space.
    By restricting the algebra to $\g s \g l(a+1,\reels)/\g g\g l(a,\reels) \times \{0\}$ (resp. $\{0\}\times \g s \g u(b+1)/\g u(b)$) we would obtain a totally geodesic embedding. This in particular implies that the outer space is of the form  $\g s \g l(N+1,\reels)/\g g\g l(N,\reels)$ (resp. $\g s \g u(P+1,Q)/\g u(P,Q)$). But it can't be both.

\item The real form of the inner space $\g s \g u(a+1)/\g u(a) \times \g s \g l(b+1,\reels)/\g g\g l(b,\reels)$. Same remark as for $\g s \g l(a+1,\reels)/\g g\g l(a,\reels) \times \g s \g u(b+1)/\g u(b)$.

\item The split real form $\g s \g l(a+1,\reels)/\g g\g l(a,\reels) \times \g s \g l(b+1,\reels)/\g g\g l(b,\reels)$. It embeds into $\g s \g l(N+1,\reels)/\g g\g l(N,\reels)$.

\item The real form $\g s \g l(a+1,\complexes)/\g g\g l(a,\complexes)$.
    The outer space $\g s \g u(P+1,Q)/\g u(P,Q)$ with $P:=\frac{a(a+3)}{2}$ and $Q:=\frac{a(a+1)}{2}$
    matches. Here $\g s=\g s \g l(2a+2,\complexes)$, with ordered simple roots $\alpha_1, \ldots ,\alpha_{2a+1}$ and $\alpha_0=\alpha_{a+1}$. $\tau$ exchanges $X_{\alpha_i}$ and $X_{\alpha_{2a+2-i}}$.

\end{enumerate}

{\sc \small e-mail: thomas.krantz@uni-greifswald.de\\
address: Institut für Mathematik und Informatik, Walther-Rathenau-Str. 47, D-17487 Greifswald}
 
\end{document}